\documentclass{amsart}
\usepackage{amssymb,latexsym}
\usepackage{amsfonts}
\usepackage[mathscr]{eucal}
\usepackage{enumerate}
\usepackage[colorlinks,linkcolor=black,anchorcolor=black,citecolor =black]{hyperref}
\usepackage{fancyhdr}

\newtheorem{thm}{Theorem}[section]

\newtheorem{tl}[thm]{Corollary}
\newtheorem{mt}[thm]{Proposition}

\theoremstyle{remark}
\newtheorem{zj}{Remark}[section]
\theoremstyle{definition}
\newtheorem{dy}{Definition}[section]
\numberwithin{equation}{section}
\allowdisplaybreaks[4]

\renewcommand{\dim}{\operatorname{dim}}

\begin{document}
\title{A Rigid Theorem for Deformed Hermitian-Yang-Mills equation}
\author{Xiaoli Han}
\address{Xiaoli Han\\ Math department of Tsinghua university\\ Beijing\\ 100084\\ China\\} \email{hanxiaoli@mail.tsinghua.edu.cn}
\author{Xishen Jin}
\address{Xishen Jin\\ Department of Mathematics\\ Renmin University of China \\ Beijing\\ 100872\\ China\\} \email{jinxishen@ruc.edu.cn}


\begin{abstract}
In this paper, we study the deformed Hermitian-Yang-Mills equation on compact K\"ahler manifold with non-negative orthogonal bisectional curvature. We prove that the curvatures of deformed Hermitian-Yang-Mills metrics are parallel with respect to the background metric if there exists a positive constant $C$ such that $-\frac{1}{C}\omega<\sqrt{-1}F<C\omega$. We also study the self-shrinker over $\mathbb{C}^n$ to the corresponding parabolic flow. We prove that  the self-shrinker over $\mathbb{C}^n$ is a quadratic polynomial function.  We also show the similar rigid theorem for the J-equations  and the self-shrinkers over $\mathbb{C}^n$ to J-flow.
\end{abstract}
\maketitle

\section{Introduction}
Let $(X,\omega)$ be a connected compact K\"ahler manifold of complex dimension $n$ and $L$ be a holomorphic line bundle over $X$. Given a metric $h$ on $L$, we define the complex function
\[
\zeta:=\frac{(\omega-F)^n}{\omega^n},
\]
where $F$ is the curvature of the Chern connection with respect to the metric $h$. It is easy to see that the average of this function is a fixed complex number
\[
Z_{L,[\omega]}:= \int_X \zeta \frac{\omega^n}{n!}
\]
dependent only the cohomology classes on $c_1(L)$ and  $[\omega]\in H^{1,1}(X,\mathbb{R})$. Let $\theta$ denote the argument of $\zeta$ and $\hat \theta$ the argument of $Z_{L,[\omega]}$. Here the branch cut is specified.

\begin{dy}
  A Hermitian metric $h$ on $L$ is said to be deformed Hermitian-Yang-Mills metric if it satisfies
  \begin{equation}
  \label{eqn-dHYM-angle}
  \theta=\hat\theta ,
  \end{equation}
  or equivalently
  \begin{equation}
    \label{eqn-dHYM-form}
    \operatorname{\mathop{Im}}(\omega-F)^n=\tan(\hat \theta) \operatorname{\mathop{Re}}(\omega-F)^n.
  \end{equation}
\end{dy}


The deformed Hermitian-Yang-Mills equation was discovered by Marino et all \cite{MMMS} as the requirement for a D-brane on the B-model of mirror symmetry to be supersymmetric. Recently, it has been studied by Collins-Jacob-Yau \cite{CJY}, Collins-Xie-Yau \cite{CXY}, Jacob-Yau \cite{JY} and some other people. According to superstring theory, the spacetime of the universe is constrained to be a product of a compact Calabi-Yau threefold and a four-dimensional Lorentzian manifold. A `duality' relates the geometry of one Calabi-Yau manifold to another `mirror' Calabi-Yau manifold. From a differential geometry viewdpoint this might be thought of a relationship between the existence of `nice' metrics on the line bundle over one Calabi-Yau manifold and the special Lagrangian submanifolds in another Calabi-Yau manifold. In \cite{LYZ}, Leung-Yau-Zaslow showed that the deformed Hermitian-Yang-Mills equation on a line bundle corresponds to the special Lagrangian equation in the mirror.

In \cite{CXY}, Collins-Xie-Yau addressed a Chern number inequality on $3$-dimensional K\"ahler manifold $(X,\Omega,L)$ admitting a deformed Hermitian-Yang-Mills metric by considering the position and winding angle of \[
\gamma(t)=-\int_X e^{-t\sqrt{-1}\omega}ch(L).
\]
As consequences of this Chern number inequality, they obtained some algebraic stability conditions to the existence of deformed Hermitian-Yang-Mills metrics.

In \cite{CJY}, Collins-Jacob-Yau showed the existence of deformed Hermitian-Yang-Mills metrics under the assumption of the existence of a $\mathcal{C}$-subsolution with supercritical phase using the method of continuity. They also conjectured some stability-type cohomological obstructions to the existence of deformed Hermitian-Yang-Mills metrics. In particular, they got a Liouville theorem for the deformed Hermitian-Yang-Mills metrics with bounded complex Hessian on $\mathbb{C}^n$ under the assumption of supercritical phase, i.e. $\hat\theta\in(\frac{n-2}{2}\pi,\frac{n}{2}\pi)$, by applying the Evans-Krylov theory.

In \cite{P2}, Pingali viewed the deformed Hermitian-Yang-Mills equation as a generalized type of Monge-Amp\`ere equation PDE with ``non-constant coefficients'' and got some existence results to the deformed Hermitian-Yang-Mills equation for some ranges of the phase angle assuming the existence of a subsolution. In particular, Pingali generalised the work of \cite{CS} on toric varieties to the deformed Hermitian-Yang-Mills equation under some algebraic geometric stability condition and addressed a conjecture of Collins-Jacob-Yau \cite{CJY}.

In \cite{JY}, Jacob-Yau provided a necessary and sufficient criterion for the existence of dHYM metrics in the case that $X$ is a K\"ahler surface. They also introduced a line bundle version of the Lagrangian mean curvature flow and proved the convergence of the flow when $L$ is sufficiently ample and $X$ has non-negative orthogonal bisectional curvature. In particular, they required that the initial data satisfies the hypercritical phase condition so that the Evans-Krylov theory works. In \cite{HY}, Han-Yamamoto established a $\varepsilon$-regularity theorem for this flow in the semi-flat case.  We also remark that the criterion of Jacbo-Yau on K\"ahler surfaces was generalized to K\"ahler manifold of complex dimension $3$ by viewing the deformed Hermitian-Yang-Mills equation as a generalized type of Monge-Am\`ere equation in \cite{P1}.

In the first part of this paper, we mainly consider the rigidity of deformed Hermitian-Yang-Mills metrics on $(X,\omega)$ with nonnegative orthogonal bisectional curvature.

Classically, rigid theorems in K\"ahler geometry aim to show that the K\"ahler manifolds satisfying some natural curvature conditions are the known examples. A famed example is the complex projective space $\mathbb{C}P^N$ equipped with the standard Fubini-Study metric $g_{FS}$ whose bisectional curvature is positive. The Frankel conjecture states that every compact K\"ahler manifold with positive bisectional curvature is biholomorphic to $\mathbb{C}P^N$. This was established independently by Mori \cite{M} and Siu-Yau \cite{SY}. An interesting new proof of the Frankel conjecture was given by Chen-Tian \cite{CT} with the help of K\"ahler-Ricci flow. In their sense, positivity of bisectional curvature is preserved under the K\"ahler-Ricci flow and the flow converges to a K\"ahler-Einstein metric exponentially. Then a classical result of Berger \cite{B} and Golberg-Kobayashi \cite{GK} on rigidity of closed constant scalar curvature K\"ahler manifold with positive bisectional curvature implies the Frankel conjecture. The results of Berger \cite{B} and Golberg-Kobayashi \cite{GK} was generalized to compact K\"ahler manifolds with nonnegative bisectional curvature in \cite{HSW} by Howard-Smyth-Wu. In details, Howard-Smyth-Wu proved that a compact csck K\"ahler manifold with nonnegative bisectional curvature is automatically a K\"ahler-Einstein manifold in \cite{HSW}. We also remark that the csck condition above was generalized to some fully nonlinear conditions on the Ricci tensor in \cite{GLZ} and the key technology is that the Ricci tensor on K\"ahler manifold satisfies the Codazzi condition.

\begin{dy}
  A K\"ahler metric $\omega$ is said to have nonnegative orthogonal bisectional curvature, if for any orthonormal tangent frame $\{e_1,\cdots,e_n\}$ at any $x\in X$, the curvature tensor of $\omega$ satisfies that
  \[
  R_{i \bar i j \bar j}= R(e_i,\bar e_{ i},e_j,\bar e_{j})\geq 0, \text{   for all } 1\leq i,j\leq n \text{ and } i\neq j.
  \]
\end{dy}

We remark that nonnegativity of the orthogonal bisectional curvature is weaker than nonnegativity of the bisectional curvature. In fact, the former condition is satisfied by not only complex projective spaces and the Hermitian symmetric spaces, but also some compact K\"ahler manifolds of dimension $\geq 2$ whose holomorphic sectional curvature is strictly negative somewhere. We refer the readers to the work of Gu-Zhang \cite{GZ}, Mok \cite{Mok} and Siu-Yau \cite{SY}.

The main result of this paper is the following theorem.

%
%
%

\begin{thm}
\label{thm-main}
 Let $(X,\omega)$ be a K\"ahler manifold with nonnegative orthogonal bisectional curvature and $(L, h)$ be a holomorphic line bundle on $X$. Suppose $h$ solves $(\ref{eqn-dHYM-angle})$. If there exists a positive constant $C$ such that $-\frac{1}{C} \omega <\sqrt{-1}F < C\omega$, i.e. the eigenvalues $\lambda_i$ is in $(-\frac{1}{C},C)$, then $\lambda_i(i=1,\cdots,n)$ are all constants. In particular, if the orthogonal bisectional curvature is strictly positive at some point $p\in X$, then $\sqrt{-1}F=c\omega$ for some constant $c$.
\end{thm}

The deformed Hermitian-Yang-Mills equation can be seen as a complex analogy to the special Lagrangian equation. In \cite{Y1,Y2}, Yuan proved that viewed as a graph over $\mathbb{R}^n$, the entire special Lagrangian submanifold in $\mathbb{C}^n$ must be a plane under some assumptions of the Hessian of the potential function. Here our assumptions can be interpreted as the complex Hessian of the Hermitian metric $h$ on line bundle $L$. The method that we adopt to establish the main theorem is the maximal principle by choosing a suitable auxiliary function.

Another interesting equation in K\"ahler geometry and mirror symmetry is the J-equation. Given K\"ahler metrics $\omega$ and $\chi$ on $X$, the J-equation is defined as
\[
\operatorname{\mathop{Tr}}_{\omega_\varphi} \chi =c
\]
where $\omega_\varphi=\omega+ \sqrt{-1}\partial\bar \partial \varphi$ is the solution to be found. The J-equation can also be rewritten in the following version
\[
\sum_{i=1}^n\frac{1}{\lambda_i}=c
\] where $\{\lambda_1,\cdots,\lambda_n\}$ are the eigenvalues of $\chi^{-1}\omega_\varphi$ as an endomorphism on $T^{1,0}X$.
The J-equation was introduced by Donaldson \cite{Don} from the viewpoint of  moment maps. In \cite{SW}, Song-Weinkove gave a necessary and sufficient condition  for the existence of the solution of J-equation. In \cite{CS}, Collins-Sz\'ekelyhidi considered the solvability of J-equation on toric varieties by the J-flow. In \cite{ChenG}, Chen gave a new numerical condition for the solvability of J-equation.

Observed by Collins-Jacob-Yau \cite{CJY}, for any fixed $n$-tuples $(\lambda_1, \cdots, \lambda_n)$ with all $\lambda_i>0$,
\[
\lim_{k\to \infty} \sum_{i=1}^n k(\frac{\pi}{2}- \arctan(k\lambda_i))= \sum_{i=1}^n \frac{1}{\lambda_i},
\]
i.e. the J-equation can be exactly seen as the limit of the deformed Hermitian-Yang-Mills equation. So we hope that the rigid theorem above should also hold for J-equation. In fact, we obtain the following theorem.
\begin{thm}
\label{thm-J}
    If $(X,\chi)$ is a compact K\"ahler manifold with non-negative orthogonal bisectional curvature and $\omega$ is a K\"ahler metric satisfying the J-equation, then the eigenvalues of $K=\chi^{-1}\omega$ are constants. In particular, if we also assume the orthogonal bisectional curvature is strictly positive at some point, then $\omega=c \chi$ for some constant $c$.
\end{thm}

Coming back to the corresponding parabolic flows to the deformed Hermitian-Yang-Mills equation and J-equation, we want to consider the self- shrinkers of these flows. In fact, we prove the following rigid theorems for the self-shrinkers on $\mathbb{C}^n$.

\begin{thm}
\label{Thm-selfshrinker}
If $u$ is a smooth function from $\mathbb{C}^n$ to $\mathbb{R}$ and satisfies the following equation
\begin{equation}
\label{eqn:self-shrinker}
\sum_{i=1}^n \arctan \lambda_i(z)-\theta_0=\frac{1}{2}(\langle z, \partial u\rangle+\langle \bar z, \bar \partial u \rangle)- u(z),
\end{equation}
where $\lambda_i$ is the eigenvalue of $\{ u_{i\bar{j}}\}$, then $u$ is  quadratic.
\end{thm}

Any function $u$ satisfying ($\ref{eqn:self-shrinker}$) leads to an entire self-similar solution
\[
v(z,t)=-tu(\frac{z}{\sqrt{-t}})
\]
to the line-bundle mean curvature flow defined by Jacob-Yau in \cite{JY}
\[
\frac{\partial v}{\partial t}=\sum_{i=1}^n \arctan \lambda_i -\theta_0.
\]

This type of rigid results for self-shrinkers to parabolic flow have been studied in \cite{CCY,DLY,DX,HW,W1,W2}. In particular, Chau-Chen-Yuan proved the rigid theorem for entire smooth Lagrangian self-shrinkers over $\mathbb{R}^n$. Theorem \ref{Thm-selfshrinker} can be seen as a complex analogy to their results.

We also consider the self-shrinkers to the $J$-flow
\[
\frac{\partial u}{\partial t}= c- \operatorname{Tr_{\omega_\varphi}\chi}.
\]
We can prove the following rigid result for entire solutions to self-shrinkers with respect to the $J$-flow.


\begin{thm}
\label{thm-J-ss}
$u:\mathbb{C}^n\to \mathbb{R}$ is a smooth  strictly pluri-subharmonic function satisfying the following equation
\begin{equation}
\label{self-shrinker-J}
 c-\sum_{i=1}^n\frac{1}{\lambda_i(z)} =\frac{1}{2} (\langle z,\partial u\rangle + \langle\bar z,\bar \partial u\rangle )- u(z)
\end{equation}
where $\lambda_i(z)$ is the eigenvalue of $\{ u_{i\bar{j}}\}$. If the complex hessian of $u$ satisfies
\[
\{u_{i\bar{j}}(z)\}\geq \frac{\sqrt{2n-1+\delta}}{|z|}I
\]
for any $\delta>0$ as $|z|\to \infty$, then $u$ is  quadratic.
\end{thm}

Similarly, any function $u$ in the theorem above leads to an entire self-similar solution
\[
v(z,t)=-t u(\frac{z}{\sqrt{-t}})
\]
to the $J$-flow
\[
\frac{\partial v}{\partial t} = c-\sum_{i=1}^n \frac{1}{\lambda_i}.
\]

\begin{zj}
In \cite{HOW}, Huang-Ou-Wang considered the similar rigid theorem for real version $J$-type equation(Theorem 1.1 in \cite{HOW}). And they did not need the extra growing condition of the Hessian. The key observation in \cite{HOW} is that the self-shrinker equation has a good expression after the Legendre transformation.  However, there is no such complex analog Legendre transformation . This difficulty was also discussed in \cite{S-W}.
\end{zj}

We will organize this paper as following. In section 2, we give some background knowledge to the deformed Hermitian-Yang-Mills equations and prove Theorem \ref{thm-main} and Theorem \ref{thm-J}. In section 3, we prove the rest results for self-shrinkers.

{\bf Acknowledgements:} Both authors are grateful to Prof. Jiayu Li and Prof. Hikaru Yamamoto for helpful  discussions.

\section{Rigid Theorem for Deformed Hermitian-Yang-Mills Equation and J-equation}
\label{section2}
\subsection{Preliminaries}
Let $(X,\omega)$ be a compact K\"ahler manifold of complex dimension $n$ and $(L, h)$ be a holomorphic line bundle over $X$. $F$ is the Chern curvature with respect to the Hermitian metric $h$ on $L$. In local coordinates, we write
\[
\omega=\frac{\sqrt{-1}}{2}g_{i\bar j} dz^i\wedge d\bar z^j
\]
and
\[
F=\frac{1}{2} F_{i\bar j} dz^i\wedge d\bar z^j=-\frac{1}{2}\partial_i\partial_{\bar j} \log (h) dz^i\wedge d\bar z^j.
\]
Since $\omega$ is K\"ahler, we have the so-called second Bianchi equality
\[
\begin{split}
F_{i\bar j,k}&=-\partial_k\partial_{\bar j}\partial_{i}\log(h)+ \Gamma^{s}_{ik} F_{s \bar j}\\
&=-\partial_i\partial_{\bar j}\partial_{k}\log(h)+ \Gamma^{s}_{ki} F_{s \bar j}=F_{k\bar j,i},
\end{split}
\]
i.e. $F_{i\bar j,k}=F_{k\bar j,i}$.

Using $F$, we introduce a Hermitian(usually not K\"ahler) metric \[
\eta=\frac{\sqrt{-1}}{2}\eta_{\bar{k}j}dz^j\wedge d\bar z^k
\]
and an endomorphism $K$ of $T^{(1,0)}(X)$ that are defined by
\[ \eta_{\bar{k}j}=g_{\bar{k}j}+F_{\bar{k}l}g^{l\bar{m}}F_{\bar{m}j},
\]
and
\[
K:=\omega^{-1}F=g^{i\bar j}F_{k\bar j} dz^i\otimes \frac{\partial }{\partial z^k}.
\]
Then the complex-valued $(n,n)$-form $(\omega-F)^n$ can be locally written as
\begin{eqnarray*}
(\omega-F)^n&=&n!\det({g_{\bar{k}j}+\sqrt{-1}F_{\bar{k}j}})(\frac{i}{2})^n dz^1\wedge d\bar{z}^1\wedge\cdots\wedge dz^n\wedge d\bar{z}^n\\
&=& \det(g^{j\bar{k}})\det({g_{\bar{k}j}+\sqrt{-1}F_{\bar{k}j}})\omega^n\\
&=&\det{(I+\sqrt{-1}K)}\omega^n.
\end{eqnarray*}
Thus the complex function $\zeta$ can be expressed as following
\begin{equation}\label{zeta}
\zeta=\det(I+\sqrt{-1}K),
\end{equation}
and the argument of $\zeta$ is
\begin{equation}
\label{theta}
\theta=-i\log\frac{\det{(I+\sqrt{-1}K)}}{\sqrt{\det(I+K^2)}},
\end{equation}
where the endomorphism $I+K^2$ can be expressed locally as
\[I+K^2=g^{p\bar{q}}\eta_{\bar{q}l}\frac{\partial}{\partial z^p}\otimes dz^l.\]
We choose the normal coordinates around some point $p\in X$ so that the endomorphism $K$ is diagonal at $p$ with eigenvalues $\{\lambda_j\}$ for $1\leq j\leq n$, then the function $\theta$ can be expressed as
\[
\theta:= \arctan \lambda_1 +\cdots +\arctan \lambda_n.
\]
Here $\theta$ takes value in $(-\frac{n}{2}\pi, \frac{n}{2}\pi)$. Then the equation $\eqref{eqn-dHYM-angle}$ is equivalent to
\begin{equation}
\theta=\hat\theta (mod~~~ 2\pi).
\end{equation}
Hence the deformed Hermitian-Yang-Mills equation \eqref{eqn-dHYM-angle} can be written as
\begin{equation}
\label{arctan-theta}
\sum_{i=1}^n \arctan \lambda_i =\theta.
\end{equation}
As described in \cite{CXY}, we also remark the constant $\theta$ in \eqref{arctan-theta} can be obtained by considering the ``winding angle'' of
\[
\gamma(t) =\int_X e^{-t\sqrt{-1} \omega} Ch(L)
\]
as $t$ runs from $+\infty$ to $1$.

Taking the derivative of (\ref{theta}) , we get (cf. Lemma 3.3 in \cite{JY} ) the first variation of $\theta$ as following
\begin{eqnarray}\label{deltatheta}
\delta\theta=Tr((I+K^2)^{-1}\delta K).
\end{eqnarray}
Thus if we take the derivatives $\partial_j$ on both sides of the deformed Hermitian-Yang-Mills equation, we have the equality
\begin{eqnarray}\label{eqn-1st-derivative}
0=\partial_j\theta=Tr(I+K^2)^{-1}\nabla_j K)=\eta^{p\bar{q}}g_{\bar{q}l}\nabla_j(g^{l\bar{m}}F_{\bar{m}p})=\eta^{p\bar{q}}\nabla_j F_{\bar{q}p}.
\end{eqnarray}
where $\{\eta^{p\bar q}\}$ is the inverse matrix of $\{\eta_{p\bar{q}}\}$.

\subsection{Proof of Theorem \ref{thm-main}}
With notations above, we give a proof of Theorem \ref{thm-main} in this subsection.
\begin{proof}
  It is easy to see that $\log \frac{\eta^n}{\omega^n}$ is a well-defined function on $X$. Indeed, it is nothing but $\log \det(I+K^2)$ if we view the endomorphism $K$ as a matrix function on $X$.
  We use the Hermitian metric $\eta$ to define the following Laplacian on $C^{\infty}(X)$
\[
\Delta_{\eta}=\eta^{i\bar j} \nabla_i \nabla_{\bar j}: C^{\infty}(X) \to \mathbb{R},
\]
where $\nabla$ is the covariant derivative with respect to $\omega$. This operator is nothing but the linearization operator of the fully nonlinear second order operator
\[
\Theta(u)=\sum_{i=1}^{n}  \arctan\lambda_{i}(u_{\alpha\bar \beta}).
\]
We remark that the operator $\Delta_{\eta}$ is elliptic as long as $F$ is bounded.
Then we compute $\Delta_{\eta} \log \frac{\eta^n}{\omega^n}$ step by step.
  \begin{equation}
  \begin{split}
          \Delta_{\eta}\log \frac{\eta^n}{\omega^n} =&\eta^{p\bar q}(\eta^{i\bar j}\eta_{i\bar j,p})_{\bar q}\\
      =&-\eta^{p \bar q} \eta^{i \bar t} \eta^{s \bar j} \eta_{s \bar t, \bar q} \eta_{i \bar j, p} + \eta^{p \bar q} \eta^{i \bar j} \eta_{i \bar j, p \bar q}.
  \end{split}
  \end{equation}
  By the expression of $\eta$, we have
  \[
\eta_{i\bar j,p}=F_{i \bar b,p}g^{a\bar b} F_{a \bar j}+ F_{i \bar b}g^{a\bar b} F_{a \bar j,p}
\]
and
\[
\begin{split}
\eta_{i\bar j, p\bar q }=&(g_{i\bar j} + F_{i \bar t} g^{s\bar t} F_{s\bar j})_{,p\bar q}\\
=&F_{i\bar t, p}g^{s\bar t} F_{s\bar j,\bar q} +F_{i\bar t,\bar q} g^{s\bar t} F_{s\bar j,p} +F_{i\bar t,p \bar q}g^{s\bar t}F_{s\bar j} +F_{i\bar t} g^{s\bar t} F_{s\bar j, p\bar q}.
\end{split}
\]
Taking trace with respect to $\eta$, we obtain
\begin{equation}
  \begin{split}
    \Delta_{\eta}\log \frac{\eta^n}{\omega^n} =&-\eta^{p \bar q} \eta^{i \bar t} \eta^{s \bar j} \eta_{s \bar t, \bar q} \eta_{i \bar j, p} + \eta^{p \bar q} \eta^{i \bar j} (F_{i\bar t, p}g^{s\bar t} F_{s\bar j,\bar q} +F_{i\bar t,\bar q} g^{s\bar t} F_{s\bar j,p} )\\
      &+\eta^{p \bar q} \eta^{i \bar j} (F_{i\bar t,p \bar q}g^{s\bar t}F_{s\bar j} +F_{i\bar t} g^{s\bar t} F_{s\bar j, p\bar q}).
  \end{split}
\end{equation}
By the second Bianchi equality and communicating law of the covariant derivatives, we have
\[
\begin{split}
F_{i\bar t, p\bar q}&= F_{p\bar t,i \bar q}\\
&=F_{p\bar t,\bar q i} + F_{a\bar t} g^{a\bar b} R_{p\bar b i \bar q}- F_{p\bar b} g^{a\bar b} R_{a\bar t i \bar q}\\
&=F_{p\bar q,\bar t i} + F_{a\bar t} g^{a\bar b} R_{p\bar b i \bar q}- F_{p\bar b} g^{a\bar b} R_{a\bar t i \bar q}.
\end{split}
\]
Thus, we have
\begin{equation}
\label{eqn-Lap1}
\begin{split}
        \Delta_{\eta}\log \frac{\eta^n}{\omega^n}
      =&-\eta^{p \bar q} \eta^{i \bar t} \eta^{s \bar j} \eta_{s \bar t, \bar q} \eta_{i \bar j, p} + \eta^{p \bar q} \eta^{i \bar j} (F_{i\bar t, p}g^{s\bar t} F_{s\bar j,\bar q} +F_{i\bar t,\bar q} g^{s\bar t} F_{s\bar j,p} )\\
      &+\eta^{p \bar q} \eta^{i \bar j} g^{s\bar t}F_{s\bar j}(F_{p\bar q,\bar t i} + F_{a\bar t} g^{a\bar b} R_{p\bar b i \bar q}- F_{p\bar b} g^{a\bar b} R_{a\bar t i \bar q} )\\
      &+\eta^{p \bar q} \eta^{i \bar j} g^{s\bar t}F_{i\bar t}(F_{p\bar q,\bar j s} + F_{a\bar j} g^{a\bar b} R_{p\bar b s \bar q}- F_{p\bar b} g^{a\bar b} R_{a\bar j s \bar q} ).
\end{split}
\end{equation}
On the other hand, taking derivative $\nabla_p$ on the both sides of \eqref{eqn-1st-derivative}, we get
\begin{equation}
\label{eqn-2nd-derivative}
\begin{split}
  \eta^{i\bar j}F_{i\bar j, \bar q p }&=-\eta^{i\bar j}_{\mathrel{\phantom{i\bar j}},p} F_{i\bar j,\bar q}\\
  &=\eta^{i\bar t}\eta^{s\bar j} \eta_{s\bar t,p} F_{i\bar j,\bar q}\\
  &=\eta^{i\bar t}\eta^{s\bar j} F_{i\bar j,\bar q} g^{a\bar b}(F_{s\bar b,p}F_{a \bar t}+ F_{s\bar b}F_{a\bar t,p}).
\end{split}
\end{equation}
Applying \eqref{eqn-1st-derivative} and \eqref{eqn-2nd-derivative} to the equation \eqref{eqn-Lap1}, we have
\[
\begin{split}
        \Delta_{\eta}\log \frac{\eta^n}{\omega^n}
      =&-\eta^{p \bar q} \eta^{i \bar t} \eta^{s \bar j} g^{a\bar b}g^{c\bar d}(F_{s \bar b,\bar q} F_{a \bar t}+ F_{s \bar b} F_{a \bar t,\bar q}) (F_{i \bar d,p} F_{c \bar j}+ F_{i \bar d} F_{c \bar j,p}) \\
      &+ \eta^{p \bar q} \eta^{i \bar j} (F_{i\bar t, p}g^{s\bar t} F_{s\bar j,\bar q} +F_{i\bar t,\bar q} g^{s\bar t} F_{s\bar j,p} )\\
      &+\eta^{i \bar j} g^{s\bar t}F_{s\bar j}\eta^{p\bar d}\eta^{c\bar q} F_{p\bar q,\bar t} g^{a\bar b}(F_{c\bar b,i}F_{a \bar d}+ F_{c\bar b}F_{a\bar d,i})\\
      &+\eta^{p \bar q} \eta^{i \bar j} g^{s\bar t}F_{s\bar j}( F_{a\bar t} g^{a\bar b} R_{p\bar b i \bar q}- F_{p\bar b} g^{a\bar b} R_{a\bar t i \bar q} )\\
      &+\eta^{i \bar j} g^{s\bar t}F_{i\bar t}\eta^{p\bar d}\eta^{c\bar q} F_{p\bar q,\bar j} g^{a\bar b}(F_{c\bar b,s}F_{a \bar d}+ F_{c\bar b}F_{a\bar d,s})\\
      &+\eta^{p \bar q} \eta^{i \bar j} g^{s\bar t}F_{i\bar t}( F_{a\bar j} g^{a\bar b} R_{p\bar b s \bar q}- F_{p\bar b} g^{a\bar b} R_{a\bar j s \bar q} ).
\end{split}
\]
For convenience, we take the normal coordinates around some point $p\in X$ such that $g_{i\bar j}(p)=\delta_{ij}$ and $F_{i\bar j}(p)=\lambda_i \delta_{ij}$. In particular, under this coordinates system, $\eta_{i\bar j}(p)=\theta_i \delta_{ij}$ where $\theta_i =1+\lambda_i^2$. We also denote $\theta^i=1/\theta_i$.

Then at $p\in X$, we have
\begin{equation}
    \begin{split}
      \Delta_{\eta}\log \frac{\eta^n}{\omega^n}(p)=&-\theta^p \theta^i \theta^j (\lambda_i+\lambda_j)^2F_{i\bar j,\bar p}F_{j\bar i,p}\\
      &+\theta^p\theta^i F_{i\bar t,p} F_{t\bar i,\bar p}+ \theta^p\theta^i F_{i\bar t,\bar p} F_{t\bar i,p}\\
      &+2\theta^i\theta^p\theta^q \lambda_i(\lambda_p+\lambda_q) F_{p\bar q,\bar i}F_{q\bar p,i}\\
      &+2\theta^i\theta^p\lambda_i^2 R_{i\bar ip \bar p} -2\theta^i\theta^p \lambda_i\lambda_p R_{i\bar i p\bar p}.
    \end{split}
\end{equation}
Inserting the formula $\theta^i(1+\lambda_i^2)=1$ to the second and third items, we get
  \begin{align*}
          \Delta_{\eta}\log \frac{\eta^n}{\omega^n}(p)
      =&-\theta^p \theta^i \theta^j (\lambda_i^2 +2\lambda_i\lambda_{ j}+ \lambda_j^2)F_{i\bar j,\bar p}F_{j\bar i,p}\\
      &+\theta^p\theta^i \theta^j(1+\lambda_j^2) F_{i\bar j,\bar p} F_{j\bar i, p}+ \theta^p\theta^i \theta^j(1+\lambda_j^2) F_{i\bar j,\bar p} F_{j\bar i,p}\\
      &+2\theta^i\theta^j\theta^p \lambda_i\lambda_j F_{j\bar i,\bar p}F_{i\bar j,p}+ 2\theta^i\theta^j\theta^p \lambda_i\lambda_j F_{i\bar j,\bar p}F_{j\bar i,p}\\
      &+2\sum_{i<p}(\theta^i\theta^p\lambda_i^2 R_{i\bar ip \bar p} -2\theta^i\theta^p \lambda_i\lambda_p R_{i\bar i p\bar p}+\theta^i\theta^p\lambda_p^2 R_{i\bar ip \bar p})\\
      =&2\theta^i \theta^j \theta^p(1+\lambda_i \lambda_j)F_{p\bar j,\bar i} F_{j\bar p,i} +2\theta^i\theta^p R_{i\bar i p\bar p}(\lambda_i-\lambda_p)^2.
  \end{align*}

Since $-\frac{1}{C}\omega<\sqrt{-1}F<C\omega$, we know that $\lambda_i\lambda_j>-1$. From  $R_{i\bar i p\bar p}\geq 0$, we know that $\log \frac{\eta^n}{ \omega^n}$ is a subharmonic function on $X$. By maximal principle, we know that $\log \frac{\eta^n}{\omega^n}$ is constant. Hence, $F_{i\bar j ,p}=0$, i.e. all $\lambda_i$ are constants(maybe different for each $i$).

In particular, if $R_{i\bar i p\bar p}$ is strictly positive at some point $p_0\in M$, then $\lambda_i=\lambda_p$ for any $i\neq p$, i.e. $\sqrt{-1}F = C\omega$.
\end{proof}

\begin{zj}
Jacob-Yau \cite{JY} can prove the existence of dHYM metrics on ample line bundle over a compact K\"ahler manifold with non-negative orthogonal bisectional curvature under the assumption of hypercritical phase ($\theta\in (\frac{n-1}{2}\pi, \frac{n}{2}\pi)$) using the line bundle mean curvature flow. In particular, the hypercritical phase condition implies $\sqrt{-1}F$ is positive.
\end{zj}

More pricisely, we have the following property for K\"ahler manifold with nonnegative orthogonal bisectional holomorphic curvature which is strictly positive at some point.

\begin{mt}
If a compact K\"ahler manifold $(M,\omega)$ has nonnegative orthogonal bisectional curvature and its orthogonal bisectional curvature is strictly positive at some point $p\in M$, then the second betti number of $M$ is $1$, i.e.
\[
\dim H^{1,1}(M,\mathbb{R})=1.
\]
\end{mt}

\begin{proof}
  For any $[\alpha]\in H^{1,1}(M,\mathbb{R})$, we know that there exists a smooth function on $M$ such that
  \[
  \operatorname{\mathop{Tr}}_{\omega}(\alpha +\sqrt{-1}\partial \bar \partial u)=c
  \]
  where $c=\frac{n\int_M \alpha\wedge \omega^{n-1}}{\int_M \omega^n}$ is a constant dependent only on $[\alpha]$ and $[\omega]$. By direct computation and communicating law of covariant derivatives, we have

\[
\begin{split}
    \Delta_{\omega} \left|\alpha_u\right|^2_{\omega}&=2\langle\Delta_{\omega}\alpha_u,\alpha_u\rangle_{\omega}+2\left|\nabla_{\omega} \alpha_u\right|^2_{\omega}\\
    &=2g^{i\bar j}g^{k \bar n}g^{m\bar l}\alpha_{m\bar n}(\alpha_{i\bar j,\bar l k}+\alpha_{p\bar l} R_{i\bar p k\bar j}-\alpha_{i\bar p}R_{p\bar lk\bar j})+2\left|\nabla_{\omega} \alpha_u\right|^2_{\omega}.
\end{split}
\]
where $g_{i\bar j}$ and $\alpha_{i\bar j}$ are the coefficients with respect to $\omega$ and $\alpha_u=\alpha+\sqrt{-1}\partial\bar\partial u$. Since $g^{i\bar j} \alpha_{i\bar j}=c$, we know
\[
g^{i\bar j} \alpha_{i\bar j,k\bar l}=0.
\]
Then we have
\[
\begin{split}
    \Delta_{\omega} \left|\alpha_u\right|^2_{\omega}&=2g^{i\bar j}g^{k \bar n}g^{m\bar l}\alpha_{m\bar n}(\alpha_{p\bar l} R_{i\bar p k\bar j}-\alpha_{i\bar p}R_{p\bar lk\bar j})+2\left|\nabla_{\omega} \alpha_u\right|^2_{\omega}\\
    &=R_{i\bar i p\bar p}(\lambda_i-\lambda_p)^2 +2\left|\nabla_{\omega} \alpha_u\right|^2_{\omega}
\end{split}
\]
  where $\lambda_i$ are the eigenvalues of $\omega^{-1}\alpha_u$ as an endomorphism on $T^{1,0}M$. By our assumption $R_{i\bar i j\bar j}\geq 0$, we know $\nabla_{\omega} \alpha_u=0$, i.e. all $\lambda_i$ are constants. Since $R_{i\bar i j\bar j}$ is strictly positive at some point, we know $\lambda_i=C$ for some constant $C$, i.e. $\alpha_u=C\omega$. In conclusion, $[\alpha]=C[\omega]$, i.e. $\dim H^{1,1}(M,\mathbb{R})=1$.
\end{proof}

\begin{zj}
In \cite{GLZ}, Guan-Li-Zhang consider a uniqueness result of
\[
F(Ric)=f(\lambda_1(Ric),\cdots,\lambda_n(Ric))=0
\]
for some fully nonlinear function $F$ under the assumption that non-negative orthogonal bisectional curvature is strictly positive at some point. They require $f$ satisfies one of the following conditions
\begin{enumerate}
  \item $f$ is concave,
  \item as a function acted on the eigenvalues of $A$, $f$ satisfies that if $\forall 0\leq l\leq n$, for all $0\leq \lambda_1\leq \cdots \leq \lambda_n$ and $\lambda_i>0$, $\forall i\geq n-l+1$, there holds
  \[
  \begin{split}
  \sum_{j,k=n-l+1}^n \ddot{f}^{jk}(A) X_{j\bar j} X_{k\bar k} &+2 \sum_{n-l+1\leq j<k} \frac{\dot f^j-\dot f^k}{\lambda_j-\lambda_k}\left|X_{j\bar k}\right|^2\\
   &+\sum_{i,k=n-l+1}^n \frac{\dot f^i(A)}{\lambda_k}\left|X_{i\bar k}\right|^2 \geq 0
  \end{split}
  \]
  for every Hermitian matrix $X=(X_{j\bar k})$ with $X_{j\bar k}=0$ if $j\leq n-l$.
\end{enumerate} Here  $\cdot$ means the derivatives of $f$ about the eigenvalues of $A$.

We can check that $f$ in our theorem does not satisfy these two conditions above if we assume $-\frac{1}{C} \omega< \sqrt{-1}F< C\omega$ for $C>1$. In fact,
\begin{enumerate}
  \item if at least one of $\lambda_i<0$, then $f$ is not concave, since
\[
\ddot f(\cdots,\lambda_i+t,\cdots)|_{t=0} = \frac{-2\lambda_i}{(1+\lambda_i^2)^2} >0.
\]
  \item if we assume $\lambda_n>1$ and $\lambda_i=0$ for $i\leq n-1$, $X_{j\bar{k}}=0$ for $j\leq n-1, k\leq n-1$ and $X_{n\bar{n}}=1$, then the left-hand side of the inequality in condition (2) is $\frac{1-\lambda_n^2}{\lambda_n(1+\lambda_n^2)^2}$ which is negative.
\end{enumerate}
\end{zj}

\subsection{Proof of Theorem \ref{thm-J}}
We can prove Theorem \ref{thm-J} by computing the same auxiliary function as in the proof of Theorem \ref{thm-main}. In this subsection, we will take
\[
\eta= \frac{\sqrt{-1}}{2}\eta_{i\bar j}dz^i \wedge d \bar z^j=\frac{\sqrt{-1}}{2}\omega_{i\bar l} \chi^{k\bar l} \omega_{k \bar j} dz^i\wedge d \bar z^j
\]
where $\omega=\frac{\sqrt{-1}}{2} \omega_{i\bar j}dz^i\wedge d\bar z^j$ and $\chi=\frac{\sqrt{-1}}{2} \chi_{i\bar j}dz^i\wedge d\bar z^j$.
  We choose the normal coordinates  near a point $p\in M$ such that
  \[
  \chi(p)=\frac{\sqrt{-1}}{2}\delta_{ij}dz^i \wedge d\bar z^j \text{ and } \omega(p)=\frac{\sqrt{-1}}{2}\delta_{ij}\lambda_i dz^i \wedge d\bar z^j.
  \]
  In particular, we have $\eta_{i\bar j}(p)=\delta_{ij}\lambda_i^2$. We also denote $\Delta_{\eta}=\eta^{i\bar j} \nabla_{\bar j} \nabla_i$ where $\nabla $ is the covariant derivative with respect to $\chi$. Indeed, $\Delta_\eta$ is the linearization operator of the $J$-equation. By direct computation, we have
  \[
  \Delta_{\eta} \log\frac{\eta^n}{\chi^n}= \frac{2}{\lambda_i\lambda_j\lambda_p^2} \omega_{p\bar j,i} \omega_{j\bar p,\bar i} +\frac{2}{\lambda_i^2\lambda_p^2} R_{i\bar ip\bar p}(\lambda_i-\lambda_p)^2\geq 0.
  \]
  Same arguments as in the proof of Theorem \ref{thm-main} implies the result needed.

\begin{zj}
  Theorem \ref{thm-J} can be proved by the same argument in \cite{GLZ} since the operator
  \[
  f(A)=\sum_{i=1}^n \frac{1}{\lambda_i}
  \]
  is concave at any positive Hermitian matrix $A$. So our proof can be seen as a new proof of the result in \cite{GLZ} when the operator is the special one above.
\end{zj}

\section{Rigid Results for Self-Shrinkers}
\label{section3}

In this section we choose the suitable barrier function to prove Theorem \ref{Thm-selfshrinker}.

\begin{proof}
We denote $\Theta(z)=\displaystyle\sum_{i=1}^n \arctan \lambda_i(z)$. Then we have
\begin{equation}
\label{eqn:Theta-1st}
\Theta_{\bar i}=\eta^{k\bar l} u_{k\bar l \bar i}, \Theta_{i}=\eta^{k\bar l} u_{k\bar l i}.
\end{equation}
On the other hand, by taking derivatives on the right side of \eqref{eqn:self-shrinker}, we get that
\begin{equation}
\label{eqn:2-nd}
(\frac{1}{2}(\langle z,\partial u\rangle +\langle \bar z,\bar \partial u\rangle)-u)_{i\bar j}=\frac{1}{2}(z_k u_{ki\bar j}+\bar z_k u_{\bar k i\bar j}).
\end{equation}
Combining the equations above, we have
\[
\eta^{i\bar j}\Theta_{i\bar j}-\frac{1}{2}(\langle z,\partial \Theta\rangle+\langle \bar z, \bar \partial \Theta \rangle)=0.
\]
We construct a barrier function as follow
\[
w(r)= \varepsilon r^{2} + \displaystyle\max_{\partial B_{r_0}}\{\Theta\}
\]
where $r_0=\sqrt{n}$ and $\varepsilon\in \mathbb{R}^+$. By direct computation, we have
\begin{equation}
\label{ineqn:w-2nd}
\frac{1}{4}(w_{rr}+\frac{2n-1}{r}w_r) -\frac{r}{2} w_r =(\frac{n}{2r}-\frac{r}{2})w_r\leq 0
\end{equation}
for any $r\geq r_0$. Furthermore,
\[
\sqrt{-1}\partial \bar \partial w=\varepsilon \sum\sqrt{-1}dz^i\wedge d\bar{z}^i >0.
\]
Hence, for $z$ outside of the ball $B_{r_0}$, we have
\[
\eta^{i\bar j} w_{i \bar j} -\frac{1}{2} (\langle z,\partial w\rangle+\langle \bar z, \bar \partial w \rangle) \leq \frac{1}{4}\Delta w -\frac{1}{2}(\langle z,\partial w\rangle+\langle \bar z, \bar \partial w \rangle)\leq 0
\]
where we use the inequality \eqref{ineqn:w-2nd} and the fact that $\eta^{i\bar j}\leq I$.

So far we have
\[
\eta^{i\bar j} w_{i \bar j} -\frac{1}{2} (\langle z,\partial w\rangle+\langle \bar z, \bar \partial w \rangle) \leq \eta^{i\bar j} \Theta_{i \bar j} -\frac{1}{2} (\langle z,\partial \Theta\rangle+\langle \bar z, \bar \partial \Theta \rangle)
\]
if $\left|z\right|\geq r_0$. On the other hand, we know
\[
w(r_0)=\varepsilon r_0^2 +\displaystyle\max_{\partial B_{r_0}}\{\Theta\} \geq \Theta \text{ on } \partial B_{r_0}
\]
and
\[
w(|x|)>\Theta(x) \text{ when }|x|\to \infty
\]
since $w(|x|)\to\infty $ as $|x|\to \infty$ while $\Theta $ is bounded. Hence $w\geq \Theta$ outside $B_{r_0}$ by maximum principle. As $\varepsilon$ tends to $0$, we know that $\Theta$ attains its global maximum on $\mathbb{C}^n$ in the closure of $B_{r_0}$. By the strong maximum principle, we know $\Theta$ is a constant.

According to the equation \eqref{eqn:2-nd}, we know that $z\cdot \partial u_{i\bar j}+\bar z\cdot\bar\partial u_{i\bar j}=0$. That implies that $r(u_{i\bar j})_r=0$. Thus $u_{i\bar j}=c$ except for $0$. However, $u_{i\bar j}$ is smooth on the whole $\mathbb{C}^n$, therefore $\{u_{i\bar j}\}=\{C_{i\bar j}\}$ on $\mathbb{C}^n$ where $\{C_{i\bar j}\}$ is a constant Hermitian matrix. Then we know that the function
\[
v=u-\sum_{i,j=1}^n C_{i\bar j} z^i \bar z^j
\]
is a pluriharmonic function, i.e. $\partial \bar \partial v=0$. In particular, there exists a holomorphic $f$ on $\mathbb{C}^n$ such that
\[
u=\sum_{i,j=1}^n C_{i\bar j}z^i \bar z^j +f-\bar f.
\]
Considering the Laurent series of $f$ and the self-shrinker equation \eqref{eqn:self-shrinker}, we can prove that $f-\bar f$ is quadratic. And so is $u$.
\end{proof}

\begin{zj}
  In \cite{HY}, a rigid result for self-shrinker to the flow  is proved. However, the function(or as an Hermitian metric of the trivial line bundle) should satisfy the so-called graphical condition in \cite{HY}. Here, we do not need any condition on $u$.
\end{zj}

We can prove Theorem \ref{thm-J-ss} by choosing analogous auxilary function as in \cite{CCY}.

\begin{proof}
We denote $J(z)=-\displaystyle \sum_{i=1}^n \frac{1}{\lambda_i(z)}$. Then
\begin{equation}
\label{1-J}
J_{\bar i} =\eta^{k\bar l} u_{k\bar l \bar i}, J_i =\eta^{k\bar l} u_{k\bar l i}
\end{equation}
where $\eta^{k\bar l} =\frac{\delta_{kl}}{\lambda_{l}^2}$. Combining the equations \eqref{eqn:2-nd} and \eqref{self-shrinker-J}, we get that
\[
\eta^{i\bar j} J_{i\bar j} -\frac{1}{2}(\langle z ,\partial J\rangle+ \langle \bar z, \bar \partial J\rangle)=0.
\]
We assume that $( u_{i\bar j})\geq \frac{\sqrt{2(2n-1+\delta)}}{r} I$ if $r\geq 1$. We choose a barrier function as follow
\[
w(r) =\varepsilon r^{1+\delta} +\max_{\partial B_{1}}\{-J\}.
\]
By direct computation, we have
\begin{equation}
\label{3-J}
\frac{r^2}{2(2n-1+\delta)}\Delta \omega- \frac{r}{2}w_r = 0.
\end{equation}
 Furtheremore,
\[
(\omega_{i\bar j})=\frac{\omega_r}{2r} (\delta_{ij} +\frac{(\delta-1)\bar z^i z^j}{2r^2}) ,
\]
since for any $X=\displaystyle\sum_{i=1}^n X^i\frac{\partial}{\partial z^i}$,
\[
\sqrt{-1}\partial \bar \partial w(X,\bar X)=X^i(\delta_{ij}+ \frac{(\delta-1)\bar z^i z^j}{2r^2}) \bar X^j\geq |X|^2+\frac{\delta-1}{2}|X|^2> 0.
\]
Thus for any $z$ outside of $B_{1}$, we have
\[
\eta^{i\bar j}w_{i\bar j} -\frac{1}{2} (\langle z,\partial w\rangle+ \langle\bar z,\bar \partial w\rangle)\leq \frac{r^2}{2(2n-1+\delta)}\Delta w -\frac{1}{2} (\langle z,\partial w\rangle+ \langle\bar z,\bar \partial w\rangle)= 0
\]
where we use the fact that $\eta^{i\bar j} \leq \frac{r^2}{2n-1+\delta} I$.

So far we have
\[
\eta^{i\bar j} w_{i\bar j}-\frac{1}{2} (\langle z,\partial w\rangle+ \langle \bar z,\bar \partial w\rangle) \leq \eta^{i\bar j} (-J)_{i\bar j}-\frac{1}{2} (\langle z,\partial (-J)\rangle+ \langle \bar z,\bar \partial (-J)\rangle)
\]
if $\left|z\right|\geq 1$. On the other hand, we know on $\partial B_{1}$
\[
w=\varepsilon+\max_{\partial B_{1}} \{-J\}\geq -J
\]
and
\[
w(\left|z\right|)> -J(z) \text{ when }\left| z\right|\to \infty
\]
by our assumption on $\sqrt{-1}\partial \bar\partial u$. Hence $w\geq -J$ outside $B_{1}$ by maximum principle. As $\varepsilon\to 0$, we know that $-J$ attains its global maximum in the closure of $B_{1}$. By the strong maximum principle, we know $-J$ is a constant.

By similar argument  as in the proof of Theorem \ref{Thm-selfshrinker}, we know that $u$ is a quadratic polynomial function.
\end{proof}

A direct consequence of the rigid theorem above is the following rigid result for self-shrinkers of $J$-flow.

\begin{tl}
\label{cor-J-ss}
If $u:\mathbb{C}^n\to \mathbb{R}$ is a smooth solution to \eqref{self-shrinker-J} satisfying
\[
\{u_{i\bar j}\}\geq CI
\]
for some constant $C$, then $u$ is quadratic.
\end{tl}

\begin{zj}
  We should also remark that the rigid theorem for a class of fully non-linear second elliptic operator is proved in Theorem 1.2 \cite{W2}. However, the operator in \cite{W2} acts on the real symmetric matrices. Using the canonical relation between Hermitian matrices and real symmetric matrices, we can regard the operator $f$ acting on $n\times n$ Hermitian matrices as an operator $\tilde f$ acting on $2n\times 2n$ real symmetric matrices.
  In fact, if we write Hermitian matrix $A$ as $A=E+\sqrt{-1}F$ where $E$ and $F$ are real matrices such that $E^T=E$ and $F^T=-F$, then the real symmetric matrix $B$ related to $A$ can be taken as
  \[
  B=\left(
      \begin{array}{cc}
        E/2 & F^T/2 \\
        F/2 & E/2 \\
      \end{array}
    \right)
  \]
  Then the operator $f:A\to \mathbb{R}$
  \[
  f(A)=\sum_{i=1}^n \arctan \lambda_i(A)
  \]
  can be regarded as a operator $\tilde{f}$ on $B$

  \[
  \tilde f(B)= \frac{1}{2}\sum_{i=1}^{2n} \arctan \lambda_i(B+J\cdot B\cdot J^T)
  \]
  where
  \[
  J=\left(
      \begin{array}{cc}
        0 & I \\
        -I & 0 \\
      \end{array}
    \right)
  \]
  is the standard complex structure on $\mathbb{C}^n$. In particular, the new operator $\tilde f$ does not satisfies condition (iii) of Theorem 1.2 in \cite{W2}, i.e. for any positive $B$
  \[
  \left|\left|D\tilde f(B) \cdot B\right|\right|\leq k_2,
  \]
  where $D\tilde f(B)=(\frac{\partial\tilde{f}}{\partial B_{ij}})$, $\|\cdot\|$ is a fixed norm on the vector space of the matrices and $k_2$ is a certain constant. In particular, we consider the case $n=1$, $A=a+\sqrt{-1}c$ and we denote
  \[
  B=\left(
      \begin{array}{cc}
        a & c \\
        c & a \\
      \end{array}
    \right)
  \]
  for $a,c\in \mathbb{R}$. Then $\tilde f$ can be expressed by
  \[
  \tilde f(B)=\arctan(a).
  \]
  In particular, $D\tilde f(B)\cdot B=\frac{B}{1+a^2}$ and its norm can not be controlled by a fixed constant.
\end{zj}

%

\end{document}